\documentclass[10pt, reqno]{amsart}
\usepackage{amssymb,latexsym}
\usepackage{eucal}
\usepackage{tikz}
\numberwithin{equation}{section}

\newtheorem{thm}{Theorem}[section]
\newtheorem{cor}[thm]{Corollary}
\newtheorem{lem}[thm]{Lemma}

\newcommand{\p}{\partial}
\newcommand{\e}{\varepsilon}

\newcommand{\A}{\mathcal{A}}
\newcommand{\FF}{\mathcal{F}}
\newcommand{\ph}{\varphi}

\newcommand{\B}{\mathcal{B}}

\newcommand{\R}{\mathcal{R}}
\newcommand{\PP}{\mathcal{P}}
\newcommand{\W}{\mathcal{W}}
\newcommand{\U}{\mathcal{U}}

\newcommand{\wtl}{\widetilde}

\newcommand{\trn}{operation}
\newcommand{\Etrn}{EN-operation}
\newcommand{\CEtrn}{CEN-operation}
\newcommand{\crd}{cyclically reduced}

\begin{document}

\title[On Conjectures of  Andrews and Curtis]
{On Conjectures of  Andrews and Curtis}
\author{S. V. Ivanov}
\address{Department of Mathematics\\
University of Illinois \\
Urbana\\   IL 61801\\ U.S.A.}
 \email{ivanov@illinois.edu}
\thanks{Supported in part by the National Science Foundation,
  grant DMS  09-01782.}

\begin{abstract} It is shown that the original Andrews--Curtis
conjecture on balanced presentations of the trivial group is equivalent to its ``cyclic" version in
which, in place of arbitrary conjugations, one can use only cyclic
permutations. This, in particular, proves a satellite conjecture of
 Andrews and Curtis \cite{AC66}  made in 1966. We also consider a more restrictive
 ``cancellative" version of the cyclic Andrews--Curtis conjecture with and without stabilizations and show that the restriction does not change the Andrews--Curtis conjecture when stabilizations are allowed. On the other hand,  the restriction  makes the conjecture false  when stabilizations are not allowed.
\end{abstract}
\maketitle

\section{Introduction}

In 1965, Andrews and Curtis \cite{AC65}  put forward a conjecture
on balanced presentations of the trivial group  and indicated some interesting topological consequences of their conjecture related to the 4-dimensional and 3-dimensional Poincar\'e conjectures. Since then both the 4-dimensional and 3-dimensional Poincar\'e conjectures have been established, however,  the Andrews--Curtis conjecture remains unsettled and has become one of the most notorious  hypotheses in group theory
and low-dimensional topology. In this paper, we  show that the  Andrews--Curtis
conjecture is equivalent to its more restrictive ``cyclic" version in
which, in place of arbitrary conjugations, one can use only cyclic
permutations. This, in particular,  proves a satellite conjecture of
 Andrews and Curtis \cite[Conjecture~3]{AC66} made in 1966.

\medskip

Let $\A = \{ a_1, \dots, a_m \}$ be an
alphabet, $\A^{-1} := \{ a_1^{-1}, \dots, a_m^{- 1} \}$, where $ a_i^{-1}$ is the inverse
of a letter $a_i \in \A$, $\A^{\pm 1} := \A \cup \A^{-1}$, and $\FF(\A)$ denote the free group
over $\A$ whose nontrivial elements are considered as reduced words over $\A^{\pm 1}$.

Let $\W = (W_1, \dots, W_n)$ be an $n$-tuple of elements of  $\FF(\A)$.
Recall  that {\em Nielsen \trn s} over $\W$  have two types and are defined as follows.
\begin{enumerate}
\item[(T1)]  For some $i$, $W_i$  is replaced with its inverse $W_i^{-1}$.

\item[(T2)]   For some pair of distinct indices $i$ and $ j$, $W_i$ is replaced with $W$, where $W=  W_i W_j$ in $\FF(\A)$.
\end{enumerate}

Consider  \trn s of a  third type so that
\begin{enumerate}
\item[(T3)] For some $i$,  $W_i$ is replaced with a word $W$ such that
$W_i$ and $W$ are conjugate in $\FF(\A)$, i.e.,  $W = S W_i S^{-1}$  in $\FF(\A)$ for some  $S  \in \FF(\A)$.
\end{enumerate}

Similarly to \cite{AC65}, \cite{AC66},  \trn  s  (T1)--(T3) are called extended Nielsen \trn s, or briefly
{\em EN-operations}.

The  Andrews--Curtis conjecture  \cite{AC65}, \cite{AC66}, abbreviated as the {\em AC-conjecture},  see also \cite{BM98}, \cite{HW93}, \cite{Lu},  \cite{MMS02},
states that, {\em for every balanced group presentation}
\begin{equation}\label{e1}
\PP = \langle   \,  a_1,  \dots, a_m  \,  \|   \, R_1 , \dots, R_m  \, \rangle
\end{equation}
{\em that defines the trivial group,  the $m$-tuple $\R = (R_1,
\dots, R_m)$ of defining words  $R_1 , \dots, R_m$   can be brought to the letter
tuple $(a_1, \dots, a_m)$ by a finite sequence of \trn  s   (T1)--(T3)}.
\smallskip

Extending the terminology by dropping the quantifier,  we will say that the AC-conjecture
{ holds} for a  balanced group presentation \eqref{e1} if  the $m$-tuple
$\R= (R_1, \dots, R_m)$   can be brought to the letter
tuple $(a_1, \dots, a_m)$ by a finite sequence of \trn  s  (T1)--(T3).
\smallskip

Consider a different, ``cyclic", version of  \trn\   of type (T3) so
that
\begin{enumerate}
\item[(T3C)] For some $i$, $W_i$ is replaced  with  a cyclic permutation $\bar W_i$  of  $W_i$.
\end{enumerate}

A satellite hypothesis, made by Andrews and Curtis \cite[Conjecture~3]{AC66}
regarding their main conjecture, claims for  $m=2$ that if a pair  $\R =(R_1, R_2)$ of words can be transformed into  $(a_1,  a_2)$ by a finite sequence of operations  (T1)--(T3),
then this can also be done by a finite sequence of  operations (T1), (T2), (T3C).
More informally, one could say that arbitrary conjugations in the AC-conjecture could be replaced with  cyclic permutations. In this paper, we confirm this hypothesis by proving a more general result for all
$m \ge 2$ which is the equivalence of the AC-conjecture to what we call a cyclic version of the AC-conjecture.

Let $\W = (W_1, \dots, W_n)$ be  a tuple  of words over $\A^{\pm 1}$ such that every $W_i$ in $\W$ is either cyclically reduced, i.e., every cyclic permutation of $W_i$ is reduced, or empty.
We call such a tuple $\W$  {\em cyclically reduced}. Consider the following operations over  cyclically reduced  tuples.

\begin{enumerate}
\item[(CT1)] For some $i$, $W_i$ is replaced with $W_i^{-1}$.

\item[(CT2)]  For some pair of distinct indices $i$ and $j$, $W_i$ is  replaced with $W$, where $W$ is a cyclically reduced or empty word obtained from the product $W_i W_j$ by making cancellations and cyclic cancellations.
\end{enumerate}

\begin{enumerate}
\item[(CT3)] For some $i$,  $W_i$ is replaced with a cyclic permutation $\bar W_i$  of  $W_i$.
\end{enumerate}

Such redefined \trn s  of type (CT1)--(CT3)  are called {\em  cyclically extended Nielsen \trn s}, or, briefly, {\em \CEtrn s}. Thus, in place of  arbitrary conjugations, we can
use only cyclic permutations and we deal with cyclically reduced tuples only.

Now the cyclic version of the Andrews--Curtis conjecture,  abbreviated as {\em CAC-conjecture},  claims that,  {\em for every presentation \eqref{e1} such that \eqref{e1} defines the trivial
group and $\R = (R_1, \dots, R_m)$ is \crd ,  the $m$-tuple   $\R$ can be brought
to the letter tuple $(a_1, \dots, a_m)$ by a finite sequence of \trn s (CT1)--(CT3)}.

\smallskip

As above, we will say that the  CAC-conjecture { holds} for a  balanced group presentation
\eqref{e1} if  the     $m$-tuple
$\R= (R_1, \dots, R_m) $ is \crd\   and  can be brought to the letter
tuple $(a_1, \dots, a_m)$ by a finite sequence of \trn  s  (CT1)--(CT3).
\smallskip

The main technical result of this paper is the following.

\begin{thm}\label{t1} Suppose that  a balanced presentation \eqref{e1}  defines the trivial
group, the $m$-tuple $\R = (R_1, \dots, R_m)$ is \crd\   and the original  Andrews--Curtis conjecture holds true for all balanced presentations in ranks  $< m$.  If the original  Andrews--Curtis conjecture holds true for the balanced presentation \eqref{e1}, then the  cyclic version of the Andrews--Curtis conjecture also  holds for \eqref{e1}.
\end{thm}

As easy consequences of Theorem~\ref{t1} we will obtain the following three corollaries.

\begin{cor}\label{c1} Let $r \ge 2$ be an integer. The  original  Andrews--Curtis conjecture is true for all balanced presentations in ranks $\le r$
if and only if the cyclic version of the Andrews--Curtis
conjecture is true for all balanced presentations in ranks $\le r$.
\end{cor}

\begin{cor}\label{c2} The original  Andrews--Curtis conjecture holds true
if and only if the cyclic version of the Andrews--Curtis
conjecture holds true.
\end{cor}

\begin{cor}\label{c3} A satellite hypothesis of Andrews and Curtis \cite[Conjecture~3]{AC66}
holds true. This hypothesis claims for $m=2$ that if  $\R =(R_1, R_2)$ can be brought
to  $(a_1,  a_2)$ by a finite sequence of  operations  (T1)--(T3), then this result can also be achieved by  operations  (T1), (T2), (T3C).

More generally, if  $r \ge 2$ is an integer and  every  $m$-tuple $\R =(R_1, \dots, R_m)$,
where $2 \le m \le r$,
that defines the trivial group by \eqref{e1},  can be transformed to  $(a_1, \dots, a_m)$ by a finite sequence of  operations  (T1)--(T3), then
such  transformation  can also be done by  operations  (T1), (T2), (T3C).
\end{cor}

Recall that there is another, more general, version of the
AC-conjecture, called the {\em AC-conjecture with stabilizations},
see   \cite{BM98}, \cite{HW93},  \cite{MMS02}, in which a fourth type of \trn s,
called {\em stabilizations}, is allowed.
\begin{enumerate}
\item[(T4)]  Add (or remove) a new letter $b$, $b \not\in \A^{\pm 1}$,  both to the alphabet $\A$ and to the tuple $\R$ of defining words (when removing, $b$ and $b^{-1}$ may not occur in all other words of $\R$).
\end{enumerate}

The AC-conjecture with stabilizations has a nice geometric
interpretation due to Wright \cite{W}: The AC-conjecture with stabilizations
is equivalent to the conjecture that every finite contractible
2-complex can be 3-deformed into a point.
Putting this result together with
the  Perelman's proof \cite{P1}, \cite{P2}, \cite{P3}  of the 3-dimensional Poincar\'e conjecture,
one can see that the AC-conjecture with stabilizations is equivalent to the claim that
every finite contractible 2-complex can be 3-deformed into a spine of a  closed 3-manifold.  Further generalizations of the AC-conjecture (with or without stabilizations), motivated by a problem of Magnus's on balanced presentations of the trivial group, can be found in \cite{I06}.   Generalization of the original AC-conjecture  in quite different direction was investigated (and proved!) by Myasnikov \cite{Ma} for solvable groups and by  Borovik, Lubotzky and Myasnikov \cite{BLM} for finite groups. In particular, it follows from results of \cite{Ma}, \cite{BLM} that the natural idea to use a solvable or finite quotient of  $F(\A)$ to construct a counterexample to the AC-conjecture will necessarily fail.

It was earlier shown by the author \cite{I05} that the
AC-conjecture with stabilizations holds for a presentation \eqref{e1}
if and only if the CAC-conjecture with stabilizations
holds for \eqref{e1} (the definition of the CAC-conjecture with stabilizations is analogous and uses operations (CT4) over cyclically reduced tuples in place of (T4)). The availability of stabilizations provides substantial aid in simulating required  conjugations by compositions of cyclic permutations with operations  (CT1)--(CT2). Such a simulation does not seem to be possible when stabilizations are not available. In particular,  we are not able to prove  the equivalence of the AC-conjecture
to its cyclic version for a given presentation (which would be similar to the result of \cite{I05}  for  AC-conjecture with stabilizations). We are only able to prove a much weaker result, namely, that the absence
of a counterexample to the AC-conjecture in rank $< m$     implies
the absence of a counterexample to the CAC-conjecture in rank $\le
m$, as stated in Theorem~\ref{t1}.  As another illustration of subtlety of operations (CT1)--(CT3), we remark that operation (CT2) is not invertible in general   and  it is not clear whether  operation  (CT2)  could be reversed with a  composition of (CT1)--(CT3). Note  that (T1)--(T3) are invertible.
\smallskip

In Sect.~2, we prove Theorem~\ref{t1} and Corollaries~\ref{c1}--\ref{c3}. In Sect.~3, we  discuss one more satellite conjecture made by Andrews and Curtis \cite[Conjecture~4]{AC66}.
This  conjecture turns out to be false and we provide a counterexample based on results of Myasnikov \cite{M}.

We remark that, in 1968, Rapaport  \cite{R1}, \cite{R2}  gave a counterexample to \cite[Conjecture~2]{AC66}. Thus all satellite Conjectures~2--4 of \cite{AC66} are now resolved  with only \cite[Conjecture~3]{AC66} being true. Recall that  \cite[Conjecture~1]{AC66} is the original Andrews--Curtis conjecture in rank $m=2$.

In Sect.~4, we look at a more restrictive version of the CAC-conjecture with and without stabilizations, abbreviated as CCAC-conjecture, in which the analogue of operation (CT2) requires a complete cancellation of one of the words $W_i, W_j$ in the cyclic product  $W_i W_j$.  We will show that CCAC-conjecture with stabilizations is still equivalent to the AC-conjecture with stabilizations, whereas the CCAC-conjecture without stabilizations is false.

\section{Proofs of Theorem~\ref{t1} and Corollaries~\ref{c1}--\ref{c3} }

We start by  proving Theorem~\ref{t1}.

Let \eqref{e1} be a presentation of the trivial group, let the words
$R_1, \dots, R_m$ be \crd\ and let the AC-conjecture be true  for all presentations of rank $< m$. We need to show that if the
AC-conjecture holds for  the presentation  \eqref{e1},
then the CAC-conjecture also holds for this presentation.
To prove this, we  argue by induction
on $m \ge 1$, assuming that both  AC- and CAC-conjectures hold true  for all presentations of rank  $< m$. Note that the basis step of this induction for $m =1$
is obvious and we may assume that $m \ge 2$.

Suppose that $\sigma_1, \dots,\sigma_\ell$ are \Etrn s that are
applied to the $m$-tuple
$$
\R = (R_1, \dots, R_m)
$$
to obtain  the letter tuple $(a_1, \dots, a_m)$. Denote   $\R(0) := \R$ and $\R(k) :=\sigma_k
\dots \sigma_1(\R)$, hence,  $\R(\ell) =  (a_1, \dots, a_m)$.
Let $X \equiv Y$ denote the
literal (or letter-by-letter) equality of words $X$, $Y$ over $\A^{\pm 1}$.
We also denote
\begin{equation*}
\R(k) := (R_1(k), \dots, R_m(k)) \quad \mbox{ and} \quad \ \bar \R(k)
:= (\bar R_1(k), \dots, \bar R_m(k))  ,
\end{equation*}
where $\bar R_1(k), \dots, \bar R_m(k)$ are \crd\ words such that,
for every $i$, $\bar R_i(k)$ is obtained from $R_i(k)$ by cyclic cancellations, hence,
$$
R_i(k) \equiv  S_i(k) \bar R_i(k)  S_i(k)^{-1} ,
$$
with some words  $S_i(k)$.

By induction on $k \ge 0$, we will be proving that $\bar \R(k)$ can be
obtained from $\R$ by a sequence of \CEtrn s (CT1)--(CT3). Since $\R$ is
cyclically reduced, we have $\bar \R(0) = \R$ and
the basis step of the induction is  true.
We now address the induction step from $k$ to $k+1$.

If $\sigma_{k+1}$ is of type (T1), then we can perform an analogous
operation (CT1) over $\bar \R(k)$ and obtain $\bar \R(k+1)$.
A reference to the induction hypothesis completes this case.

If $\sigma_{k+1}$ has type (T3), then no change is needed, we can
set  $\bar \R(k+1) := \bar \R(k)$, and refer to the induction hypothesis.
\medskip

From now on assume that $\sigma_{k+1}$ has type (T2) and $R_s(k+1) =
R_s(k)R_t(k)$ in $\FF(\A)$ with $t \neq s$. To simplify notation, rename
$U_i := \bar R_i(k)$, $i =1, \dots, m$, and $\U := \bar \R(k)$.

\smallskip

Suppose that $X, Y_1, \ldots, Y_r$ are words over $\A^{\pm 1}$. We say that
$X$ {\em occurs} in words $Y_1, \ldots, Y_r$ if there an index  $i$ such that
$Y_i \equiv Y_{i,1} X Y_{i,2}$ with some words $Y_{i,1}, Y_{i,2}$, i.e.,
$X$ occurs in at least one of the words $Y_1, \ldots, Y_r$.
In this case, we may also say that $X$ is a {\em subword} of  $Y_1, \ldots, Y_r$.
The number of occurrences of  $X$ in  $Y_1, \ldots, Y_r$ is the sum of
the numbers  of occurrences of  $X$ in every $Y_i$.

\begin{lem}\label{Lem1} For every letter $a_j \in \A$ and every word $U_i $ in $\U$,
one may assume that the number of  occurrences of $a_j$ in words $U_i, U_i^{-1}$ is not one.
\end{lem}

\begin{proof} Suppose, on the contrary,
that the words $U_i, U_i^{-1}$ contain a single occurrence of a letter
$a_j \in \A$. Reindexing and using
operations (CT1), (CT3) if necessary, we may assume that $i = j
=1$ and $U_1 \equiv a_1 U_{1,0}$,
where $U_{1,0}$ has no occurrences of $a_j$ and $a_j^{-1}$.
Applying operations (CT1)--(CT3),  we can turn $\U$ into $(U_1, V_2,
\dots, V_m)$, where $V_2, \dots, V_m$ have no occurrences of
letters $a_1, a_1^{-1}$. Hence, the presentation $\langle a_2,
\dots, a_m \, \| \,  V_2,\dots, V_m \rangle$ defines the trivial
group and, by the induction hypothesis on $m$, the $(m-1)$-tuple
$(V_2, \dots, V_m)$ can be transformed into $(a_2, \dots, a_m)$ by
\CEtrn s (CT1)--(CT3). Consequently, using \CEtrn s  (CT1)--(CT3),
we can  turn the $m$-tuple $\U$ into $(U_1, a_2, \dots, a_m)$
and, hence, into $(a_1, a_2, \dots, a_m)$. The proof is complete.
\end{proof}

Reindexing if necessary, we may also assume that $s=1$ and $t =2$,  hence
$R_1(k+1) = R_1(k)R_2(k)$  in $\FF(\A)$. Let $|W|$ denote the length of a word $W$.

\begin{lem}\label{Lem2} Up to cyclic permutations of the words
$\bar R_1(k+1)$,  $U_1$, $U_2$, one may assume that the \crd\  word $\bar R_1(k+1)$,
conjugate to $R_1(k+1)$ in $\FF(\A)$, has one of the following four forms (F1)--(F4),
depicted in Figs.~1(a)--(c).

\begin{enumerate}
\item[(F1)] $\bar R_1(k+1) \equiv U_1 C U_2 C^{-1}$, where $|C| >
0$, see  Fig.~1(a).

\item[(F2)] $\bar R_1(k+1) \equiv DE$, where $U_1 \equiv D P^{-1}$,
$U_2 \equiv P E$, and  $|D|, |E| >0$,    see Fig.~1(b).

\item[(F3)] $\bar R_1(k+1) \equiv D$ and $U_1 \equiv D C U_2^{-1}
C^{-1}$,   where $|C| \ge
0$, see  Fig.~1(c).

\item[(F4)] $\bar R_1(k+1) \equiv D$ and $U_2 \equiv D C U_1^{-1}
C^{-1}$,   where $|C| \ge 0$,  see  Fig.~1(c).
\end{enumerate}
\end{lem}

\begin{center}
\begin{tikzpicture}[scale=.88]
\draw  (-4.5,4) ellipse (1 and .8);
\draw  (-1.5,4) ellipse (1 and .8);
\draw  (-3.5,4) [fill = black] circle (.05);
\draw  (-2.5,4) [fill = black] circle (.05);
\node at (-5,4) {$U_1$};
\node at (-1,4) {$U_2$};
\node at (-3,4.5) {$C$};
\draw [-latex](-0.5,4.1) --(-0.5,3.95);
\draw [-latex](-5.5,3.95) --(-5.5,4.1);
\draw (-3.5,4) --(-2.5,4);
\node at (-3,2) {Fig.~1(a)};
\node at (2.5,2) {Fig.~1(b)};
\node at (7,2) {Fig.~1(c)};

\draw  (2.5,4) ellipse (1 and 1.2);
\draw [-latex](-3.5,4) --(-2.9,4);
\draw (2.5,2.8) --(2.5,5.2);
\draw [-latex](2.5,2.8) --(2.5,4.1);
\draw  (2.5,2.8) [fill = black] circle (.05);
\draw  (2.5,5.2) [fill = black] circle (.05);
\draw [-latex](3.5,4.1) --(3.5,3.95);
\draw [-latex](1.5,3.95) --(1.5,4.1);
\node at (1,4) {$D$};
\node at (2,4) {$P$};
\node at (3,4) {$E$};

\draw  (7,4) ellipse (1.5 and 1);
\draw  (6.9,4) ellipse (.85 and .6);
\draw (8.5,4) --(7.75,4);
\draw [-latex](8.5,4) --(8,4);
\draw  (8.5,4) [fill = black] circle (.05);
\draw  (7.75,4) [fill = black] circle (.05);
\draw [-latex](6.05,3.92) --(6.05,4.07);
\draw [-latex](7,5) --(7.1,5);
\node at (6.9,4) {$U_2 (U_1)$};
\node at (7.99,4.37) {$C$};
\node at (7,5.4) {$D$};
\end{tikzpicture}
\end{center}

\begin{proof} It follows from the definitions that, letting $S_1 := S_1(k)$ and $S_2 := S_2(k)$,
we have the following equalities
\begin{align*}
& R_1(k) \equiv S_1 U_1 S_1^{-1} ,  \quad  R_2(k) \equiv S_2 U_2 S_2^{-1} , \quad  R_1(k+1)  \equiv S_1 U_1 S_1^{-1}  S_2 U_2 S_2^{-1} , \\
& R_1(k+1) \equiv S_1(k+1) \bar  R_1(k+1) S_1(k+1)^{-1} .
\end{align*}

To analyze possible cancellations in the product  $S_1 U_1 S_1^{-1}  S_2 U_2 S_2^{-1}$, we consider a disk diagram $\Delta$ over the group presentation
\begin{equation}\label{prU2}
\PP_U = \langle \,  a_1,  \dots, a_m  \,  \|  \, U_1 , U_2  \, \rangle  .
\end{equation}
Recall that a disk diagram over a group presentation is a finite connected and simply connected 2-complex with a labeling function used for geometric interpretation of consequences of defining relations, details can be found in \cite{I94}, \cite{LS}, \cite{Ol}. We define
a disk diagram $\Delta$ over  \eqref{prU2} so that $\Delta$ contains two faces $\Pi_1, \Pi_2$ whose clockwise oriented boundaries  $\p \Pi_1, \p \Pi_2$  are labeled by  words $U_1, U_2$, resp., and $\Delta$ contains a vertex $o$ that is connected to  $\p \Pi_1$, $\p \Pi_2$ by paths labeled by words $S_1, S_2$, resp., see Fig.~2. Then the clockwise oriented boundary $\p{|_o}\Delta$ of $\Delta$, starting at $o$, is labeled by  the word $S_1 U_1 S_1^{-1}  S_2 U_2 S_2^{-1} $, which we write in the form
$ \ph( \p{|_o}\Delta ) \equiv  S_1 U_1 S_1^{-1}  S_2 U_2 S_2^{-1}$.

\begin{center}
\begin{tikzpicture}[scale=.70]
\draw  (-4.5,4) ellipse (1 and 1);
\draw  (-0.5,4) ellipse (1 and 1);
\draw  (-2.5,2) [fill = black] circle (.05);
\node at (-6,4) {$U_1$};
\node at (1,4) {$U_2$};
\node at (-3.6,2.4) {$S_1$};
\node at (-1.36,2.4) {$S_2$};
\draw [-latex](0.5,4.1) --(0.5,3.95);
\draw [-latex](-5.5,3.95) --(-5.5,4.1);
\node at (-2.5,.7) {Fig.~2};
\node at (-2.5,1.5) {$o$};
\draw (-2.5,2) --(-3.8,3.3);
\draw (-2.5,2) --(-1.2,3.3);
\draw [-latex](-2.5,2) --(-3.25,2.75);
\draw [-latex](-2.5,2) --(-1.75,2.75);
\draw  (-1.2,3.3) [fill = black] circle (.05);
\draw  (-3.8,3.3) [fill = black] circle (.05);
\draw [-latex](-5.5,3.95) --(-5.5,4.1);
\draw [-latex](-5.5,3.95) --(-5.5,4.1);
\end{tikzpicture}
\end{center}

\noindent
Cancellations in the cyclic  word  $ S_1 U_1 S_1^{-1}  S_2 U_2 S_2^{-1}$ can be interpreted
as folding and pruning off edges in $\Delta$ which, after all cancellations are done,
will turn into a diagram $\Delta'$ that contains two faces $\Pi'_1, \Pi'_2$ and has
 $\ph( \p{|_{o'}}\Delta' ) \equiv \bar R_1(k+1)$ for a suitable vertex
 $o' \in  \p \Delta'$. It is not difficult to check that the following hold true.
 If the boundaries $\p \Pi'_1$, $\p \Pi'_2$ have no common vertex then Case (F1) holds. If the boundaries  $\p \Pi'_1$, $\p \Pi'_2$ contain a common vertex and both  $\p \Pi'_1$, $\p \Pi'_2$ have  edges on $\p \Delta'$, then Case        (F2) holds. If $(\p \Pi'_2)^{-1}$ is a subpath of $\p \Pi'_1$, then Case (F3) holds.
Finally, if $(\p \Pi'_1)^{-1}$ is a subpath of $\p \Pi'_2$, then Case (F4) holds.
\end{proof}

In view of Lemma~\ref{Lem2}, we need to consider Cases (F1)--(F4).
\smallskip

In Case (F2),  we apply (CT3) to $U_1 $ to get $DP^{-1}$ and apply (CT3) to $U_2$ to get $PE$. Then we use (CT2) to turn  $DP^{-1}$ into $DE$. Since $DE$ is a cyclic permutation of $\bar R_1(k+1)$, the induction step is complete.

\smallskip

In Case (F3),  we apply (CT3) to $U_1 $ to get $C^{-1} D C U_2^{-1}$ and use (CT2) to
make the transformation  $C^{-1} D C U_2^{-1} \to D$.  Since $D$ is a cyclic permutation of $\bar R_1(k+1)$, the induction step is complete.
\smallskip

Case (F4) is analogous to Case (F3) with $U_1$ and $U_2$ switched.

\smallskip

It remains to study Case (F1).

Let $\W = (W_1, \dots, W_m)$ be an $m$-tuple of nonempty \crd\
words over $\A^{\pm 1}$. Let $W_1 \equiv AB$, where $|A|, |B| >0$, and $C$  be a
cyclic permutation of the word $W_{j'}^\e$, where $j' >1$ and $\e =
\pm 1$. If the word $ACB$ is reduced, then the operation over $\W$
so that $W_1$ is replaced with $ACB$ is called a {\em simple 1-insertion}.
Let $\W(i)$ denote an $m$-tuple obtained from a \crd\ tuple $\W = \W(0)$ by a sequence of $i$ simple
1-insertions. It is easy to see that $\W(i)$ can be obtained from $\W(0)$
by a sequence of \CEtrn s (CT1)--(CT3) and that $\W(0)$ can be
obtained back from $\W(i)$  by a sequence of \CEtrn s
(CT1)--(CT3).
\medskip

Consider the following properties of a reduced word  $W$ over the alphabet $\A^{\pm 1}$.

\begin{enumerate}
\item[(P)] For a given letter $a \in \A^{\pm 1}$,  there are  distinct letters $b, c \in \A^{\pm 1}$ such that
$ab$ and $ac$ occur in the words $W, W^{-1}$.

\item[(Q)]  For every letter $a \in \A^{\pm 1}$, there are  distinct letters
$b, c \in \A^{\pm 1}$, depending on $a$, such that
$ab$ and $ac$  occur in the words $W, W^{-1}$.
\end{enumerate}

\begin{lem}\label{Lem3}
Suppose the first and the last letters of the word $U_1$ are distinct.   Then there exists a sequence of simple
1-insertions that transform the tuple
$$
\U = (U_1, \dots, U_m)
$$
into $\U_V = (V, U_2, \dots, U_m)$ in which the \crd\ word
$V$ has  property (Q).
\end{lem}

\begin{proof} Note that, by Lemma~\ref{Lem1}, $|U_i | >1$ for every $i$.
Next, we observe that if $U_1 \equiv AB$ with $|A|, |B| >
0$, then for every $j \ge 2$ at least one of the words $AU_jB$,
$AU_j^{-1}B$ is \crd .  We also note that if a letter $e \in \A$ does not occur
in words  $U_2, U_2^{-1}, \dots$,  $ U_{m}, U_m^{-1}$  then it follows from
Lemma~\ref{Lem1} and the triviality of
the group given by presentation
\begin{equation}\label{gpau}
\langle \, a_1, \ldots , a_m  \,  \|  \,  U_1, \ldots , U_m \,  \rangle
\end{equation}
that $U_1$ contains at least three occurrences of
$e, e^{- 1}$ and, by Lemma's assumption, at most one of these occurrences  is the first or the last letter of $U_1$. On the other hand, if a letter $e \in \A$ occurs in words $U_i, U_i^{-1}$ for some $i >1$ then,  by Lemma~\ref{Lem1}, there are at  least two
such occurrences.  Therefore, using  the words  $U_2, \dots$,
$U_{m}$ and making at most  $m-1$ simple
1-insertions, we can obtain a word $V_{m-1}$ from $V_0 := U_1$  such
that, for every $a \in \A^{\pm 1}$, the words $V_{m-1}$,
$V_{m-1}^{-1}$ contain at least two distinct occurrences of words
$ab$, $ac$, where $b, c \in \A^{\pm 1}$. Below we will make more simple  1-insertions to guarantee    that $b \ne c$, i.e., to guarantee that the resulting word would have property (P) relative to $a$.

\smallskip

Fix a letter $a \in \A^{\pm 1}$. As was shown above, the words $V_{m-1}$,
$V_{m-1}^{-1}$ contain  two distinct occurrences of words
$ab$, $ac$, where $b, c \in \A^{\pm 1}$. If $b \ne c$ then $V_{m-1}$ has property (P) relative to $a$ and we do not need to do anything. Suppose $b = c$.

\smallskip

First assume that the letter $a$ does not occur in words
 $U_2 , U_2^{- 1} , \dots$, $U_m, U_m^{-1}$.

\smallskip

Let $V_{m-1}^{\delta} \equiv DE$ be the factorization of
$V_{m-1}^{\delta}$, $\delta = \pm 1$, defined by an
occurrence of  $ab$  in  $V_{m-1}^{\delta}$, so that  $D \equiv D_1 a$ and $E \equiv b E_1$. It
follows from Lemma~\ref{Lem1} and the triviality of the group given by presentation \eqref{gpau}
that $U_2, U_2^{- 1}$ contain occurrences  of at
least two distinct letters of $\A$.
This means that if $b$   occurs in $U_2$, $U_2^{- 1}$, then there is a cyclic
permutation $\wtl U_2^\e$ of $U_2^\e$, $\e = \pm 1$, such that
$\wtl U_2^\e \equiv b^{- 1} F d$, where $d \in \A^{\pm 1}$,    $d \neq b^{- 1}$. Then
$
V_{m-1} \to V_{m} \equiv (D \wtl U_2^{\e} E)^\delta
$
is a simple 1-insertion that creates a subword $a b^{-1}$ in $V_{m}$, $V_{m}^{-1}$.
Hence, both $a b^{-1}$, $ab$ occur in $V_{m}$, $V_{m}^{-1}$, as desired.
On the other hand, if $b$ does not occur in $U_2, U_2^{- 1}$,
then
$
V_{m-1} \to V_{m}
\equiv (D U_2 E)^\delta
$
is a simple 1-insertion that
produces a subword $ad$ in $V_{m}^{\delta}$, for some $d \in \A^{\pm 1}$,
$d \neq b$. Hence, both $a d$, $ab$ occur in $V_{m}$, $V_{m}^{-1}$, as required.

The case when $a$ does not occur in words
$U_2 , U_2^{- 1} , \dots$, $U_m, U_m^{-1}$ is complete.
\smallskip

Now assume that the letter $a$ occurs in words $U_2 , U_2^{- 1} , \dots$, $U_m, U_m^{-1}$, say, $a$ occurs in
$U_2, U_2^{- 1}$. As above, we
let $V_{m-1}^{\delta} \equiv DE$   be the factorization of
$V_{m-1}^{\delta}$, $\delta = \pm 1$, defined by an
occurrence of  $ab$  in  $V_{m-1}^{\delta}$,
so that  $D \equiv D_1 a$ and $E \equiv b E_1$.
\smallskip

Consider two cases: $b = a$ and $b \neq a$.

Assume that $b = a$. It follows from  Lemma~\ref{Lem1}  and the triviality of the group
given by presentation \eqref{gpau} that $U_2$ is not a power of $a$. Hence,
there is a cyclic permutation $\wtl U_2^\e$ of
$U_2^\e$, $\e = \pm 1$, such that $\wtl U_2^\e \equiv d F a$, $d \in \A^{\pm 1}$   and    $d
\neq a^{\pm 1}$. Then
$
V_{m-1} \to V_{m} \equiv (D \wtl
U_2^{\e} E)^\delta
$
is a simple 1-insertion that creates a
subword $ad$ in $V_{m}^{\delta}$ with $d \neq b$.
This means that  $a b$ and $ad$ occur in $V_{m}$, $V_{m}^{-1}$, as required.
\smallskip

Suppose that $b \ne a$. Recall that $a$ occurs in $U_2, U_2^{-1}$.
If $b$ does not occur in $U_2, U_2^{-1}$, then we consider a cyclic
permutation $\wtl U_2^\e$ of $U_2^\e$, $\e = \pm 1$, such that
$\wtl U_2^\e \equiv a F$. Then
$
V_{m-1} \to V_{m} \equiv (D \wtl U_2^{\e } E)^\delta
$
is a simple 1-insertion that creates a subword $a^2$ in $V_{m}^{\delta}$.
Hence, $a b$ and $aa$ occur in $V_{m}$, $V_{m}^{-1}$, as desired.
On the other hand, if $b$ occurs in $U_2, U_2^{-1}$, then there is a cyclic
permutation $\wtl U_2^\e$ of $U_2^\e$, $\e = \pm 1$, such that
$U_2^\e \equiv b^{-1} F d$, $d \in \A^{\pm 1}$   and $b^{-1} \neq d$. Then
$
V_{m-1} \to V_{m} \equiv (D \wtl U_2^{\e} E)^\delta
$ is a simple
1-insertion that creates a subword $ab^{-1}$ in
$V_{m}^{\delta }$.
This means that $a b$ and  $ab^{-1}$ occur in $V_{m}$, $V_{m}^{-1}$, as required.
\medskip

Let us overview our intermediate findings.
When given a letter $a \in \A^{\pm 1}$ and  two distinct occurrences
of $ab$  in $V_{m-1},  V_{m-1}^{-1}$, where $b\in \A^{\pm 1}$,  we are always able to make  a simple 1-insertion
$
V_{m-1} \to V_m \equiv  (D \wtl U_j^{\e} E)^\delta
$
in such a way that $j >1$, $\e = \pm 1$, $\delta = \pm 1$,  and $V_{m-1}^{\delta} \equiv DE$ is the factorization of
$V_{m-1}^{\delta}$,  defined by an occurrence of
$ab$ in $V_{m-1}^{\delta}$, so that  $D \equiv D_1 a$ and $E \equiv b E_1$. Moreover, if $\wtl U_j^{\e}$ starts with a letter $d \in  \A^{\pm 1}$ then $d \ne b$. In this situation, we say that the subword $ab$ of $V_{m-1}^{\delta} \equiv D_1ab E_1$ was changed by the  simple 1-insertion  $V_{m-1} \to V_m$  into  $ad$.
Since $d \ne b$ and both $ab, ad$ occur in $V_{m}$, $V_{m}^{-1}$, it follows that the word  $V_{m}$ has  property (P) relative to  $a$.

Clearly, for every $a' \in \A^{\pm 1}$,    the words $V_{m}$, $V_{m}^{-1}$,  similarly  to $V_{m-1}$, $V_{m-1}^{-1}$,  also contain at least two distinct occurrences of subwords of the form
$a'b'$, $a'c'$, where $b', c' \in \A^{\pm 1}$. If $V_{m}$ does not have property (P)  relative to $a'$, then $b' = c'$ and,  arguing as above, we can make a  simple 1-insertion that converts  $V_{m}$ into  $V_{m+1}$ and changes one of  the
subwords $a'b'$    into $a'd'$ with $d' \ne b'$. Note that this simple 1-insertion  preserves   property (P) of   $V_{m}$ relative to $a$ because one of the occurrences of $a' b'$ is not affected by the performed   1-insertion. As a result, the word  $V_{m+1}$ has  property (P)  relative to $a$ and relative to $a'$.
Iterating our arguments for all  $e \in \A^{\pm 1}$, we obtain  property (P) for the word $V :=  V_{m'}$, where $m' \le 3m-1$, relative to  every letter $e \in  \A^{\pm 1}$ which means property (Q) for $V$.
\end{proof}

\begin{lem}\label{Lem4} Suppose that the first word $U_1$ of the
$m$-tuple $\U =(U_1, \dots, U_m)$  has property (Q). Furthermore, assume that the
word $U'_1 C U_2C^{-1}$, where $C$ is some word
and $U'_1$ is a cyclic permutation of $U_1^{\e_1}$, $\e_1
=\pm 1$, is \crd . Then the $m$-tuple $\U$ can be transformed into
$$
\U_C = (U'_1 C U_2C^{-1}, U_2, \dots, U_m)
$$
by a finite sequence of \CEtrn s (CT1)--(CT3).
\end{lem}

\begin{proof} First we will show that
it suffices to prove that $\U$ can be turned into
\begin{equation}\label{eq3}
     \U'_C  = (U_1, CU_2C^{-1} U_1', U_3, \dots, U_m)
\end{equation}
by operations  (CT1)--(CT3). Indeed, starting with $\U'_C$ and
using (CT1)--(CT2), we can switch $U_1$ and $CU_2C^{-1} U_1'$. Note that a switch is a composition of  (CT1)--(CT2).  Then we  obtain $U_1'$ in place of $U_1$ by (CT3) and (CT1) if $\e_1 = -1$, and  multiply $U_1'$ on the right by $(CU_2C^{-1} U_1')^{-1}$ to get $U_2^{-1}$ in place of $U_1'$.
As a result, we obtain the $m$-tuple
$$
( CU_2C^{-1} U_1', U_2^{-1}, U_3, \dots, U_m)
$$
which can be easily converted into $\U_C$ by (CT1), (CT3).
\smallskip

To prove  that the tuple $\U'_C$, defined by  \eqref{eq3},  can be obtained from $\U$ by \CEtrn s
(CT1)--(CT3), we will argue  by induction on the length $|C| \ge 0$ of the word $C$.

If $|C| = 0$, then our claim is obvious. Assume that $|C| > 0$ and let
$$
C \equiv a C_1 ,
$$
where $a \in \A^{\pm 1}$. By property (Q) for $U_1$, there are distinct letters
$b, c \in  \A^{\pm 1}$ such that  $a^{}b, a^{}c$ occur in words
$U_1$, $U_1^{- 1}$.  Consequently, there is a cyclic
permutation $U_{1, b}$ of  $U_1$ or $U_1^{-1}$  and there is a cyclic
permutation $U_{1, c}$ of  $U_1$ or $U_1^{-1}$ such
that
$$
 U_{1, b} \equiv a^{-1} D_b b^{-1} , \qquad  U_{1, c} \equiv a^{-1} D_c c^{-1} ,
$$
and at least one of the words $C_1 U_2 C_1^{-1} U_{1, b}$, $C_1 U_2 C_1^{-1} U_{1,
c}$ is \crd . Let $U_0$ denote one of  $U_{1,
b}$, $U_{1, c}$ for which the product $ C_1 U_2 C_1^{-1}U_{0 } $
is \crd .  We also denote $U_{0} \equiv a^{-1} U_{0,1}$. By the
induction hypothesis on $|C|$, $\U$ can be converted into
\begin{equation}\label{eq4}
     \U'_{C_1}  := (U_1, C_1 U_2 C_1^{-1} U_0, U_3, \dots, U_m)
\end{equation}
by a sequence of operations (CT1)--(CT3).
Using more  operations (CT1)--(CT3), we can cyclically permute
$C_1 U_2 C_1^{-1} U_0$ to obtain the word $U_{0, 1} C_1 U_2
C_1^{-1}a^{-1} $ and multiply it by $U_1'$ on the left
(which is a composition of (CT1)--(CT2)) to obtain the word  $U_1' U_{0, 1}
C_1 U_2 C_1^{-1}a^{-1}$.
Consider the word
\begin{equation}\label{dop1}
 a^{-1}   U_1' U_{0, 1} C_1 U_2 C_1^{-1} .
\end{equation}

Since $C_1  U_2 C_1^{-1} a^{-1} U_{0, 1}$ is \crd ,
it follows that a first cancellation in the  word \eqref{dop1}, if it exists,
occurs in the prefix subword  $a^{-1} U_1' U_{0, 1}$.
We now discuss cancellations in this subword.

Let $T$ denote an  empty or reduced word obtained from $a^{-1}U_1' U_{0, 1}$ by cancellations.
Since $U_1' C U_2 C^{-1}$ is \crd\  and $C
\equiv a C_1$, it follows that $U'_1$ does not start with $a$.
We also note that the words $ U_1'$, $ U_{0, 1}$ are reduced, and $|U_1'| = |U_{0, 1} |+1$.
Hence, we may conclude $|T| > 0$. Observe that either of $U_1'$, $ U_{0, 1} a^{-1}$ is a cyclic permutations of $U_1$ or $U_1^{-1}$.
Hence, the word $a^{-1}U_1' U_{0, 1}$ represents the trivial element in the
one-relator group given by presentation
\begin{equation}\label{gpU1}
\langle \, a_1, \ldots , a_m  \,  \|  \,  U_1  \,  \rangle .
\end{equation}
Since the word $U_1$ has property (Q), it follows that  every letter
of $\A$ occurs in $U_1, U_1^{-1}$ at least twice. By the classical Magnus's Freiheitssatz, see \cite{LS}, \cite{MKS}, for  one-relator group presentation \eqref{gpU1} applied to the word $T$, we obtain that every letter
of $\A$ must occur in $T, T^{-1}$. Since either of $U_1'$, $U_{0, 1} a^{-1}$ is a cyclic permutation of
$U_1$ or $U_1^{-1}$, it follows that every letter of $\A$  occurs in
$a^{-1}U_1' U_{0, 1}$,   $(a^{-1}U_1' U_{0, 1})^{-1}$   even and positive number of times. Since cancellations of letters are done in pairs,  $T = a^{-1}U_1' U_{0, 1}$ in  $\FF(\A)$,   and  every letter
of $\A$ occurs in $T, T^{-1}$, we deduce that  every letter of $\A$ occurs in $T, T^{-1}$  at least twice and so
$|T| \ge 2m \ge 4$. This implies that  both the first and the last letters of $a^{-1}U_1' U_{0, 1}$ remain uncanceled in the reduced word $T$. Therefore, we may conclude that the word $TC_1 U_2 C_1^{-1}$ is \crd .

Now we  apply a sequence of  \CEtrn s (CT1)--(CT3) to
the tuple   $\U_{C_1}'$, defined by  \eqref{eq4},   so that the second component of $\U_{C_1}'$    would be changing as follows:
\begin{multline*}
C_1 U_2 C_1^{-1} U_0 \to a^{-1} U_{0, 1} C_1 U_2 C_1^{-1} \to
a^{-1} U_1' U_{0, 1} C_1 U_2 C_1^{-1} \overset{\FF(\A)}  = \\
 T  C_1 U_2 C_1^{-1} \to T U_0^{-1}   C_1 U_2 C_1^{-1} \overset{\FF(\A)} =
a^{-1}U_1'  U_{0, 1}  (a^{-1} U_{0, 1})^{-1}   C_1 U_2 C_1^{-1} \overset{\FF(\A)} =
\\  a^{-1}U_1' a C_1 U_2 C_1^{-1} \to  U_1' a C_1 U_2 C_1^{-1}
a^{-1} \equiv   U_1' C U_2 C^{-1}  \to     C U_2 C^{-1} U_1'  .
\end{multline*}
The last word is a desired one and the induction step is complete.
\end{proof}

Recall that, to prove Theorem~\ref{t1}, it remains to study Case (F1), that is,
to  establish that the
operation $U_1 \to U_1 CU_2 C^{-1}$, where $ U_1 CU_2 C^{-1}$ is \crd ,
over the $m$-tuple $\U$ is a composition of
\CEtrn s (CT1)--(CT3). To do this, we first apply Lemma~\ref{Lem3}
and, using simple 1-insertions, turn $U_1$ into a word $V$
with property (Q). Note that  a simple
1-insertion, by the definition, does not change the first
and the last letters of $U_1$ and so the word $VCU_2C^{-1}$, similarly to $U_1 CU_2C^{-1}$,  is \crd . Therefore, Lemma~\ref{Lem4} applies and yields a
sequence of \CEtrn s (CT1)--(CT3) that transforms the $m$-tuple
$\U_V = (V, U_2, \dots, U_m )$ into $(VCU_2C^{-1}, U_2, \dots, U_m
)$. Now we can use   \CEtrn s that convert the subword $V$ of
$VCU_2C^{-1}$ back into $U_1$ (these can be viewed as inverses of
simple 1-insertions). As a result,  we obtain the desired
tuple $(U_1 CU_2C^{-1}, U_2, \dots, U_m)$ and Case (F1) is complete. Theorem~\ref{t1} is proved. \qed
\medskip

\begin{proof}[Proof of Corollary~\ref{c1}] Let $r \ge 2$ be an integer. Suppose that the AC-conjecture holds for every presentation \eqref{e1} of rank $\le r$.  Then, by  induction on $m$, where $1 \le m \le r$, it follows from Theorem~\ref{t1} that the CAC-conjecture also holds for every presentation \eqref{e1} of rank $\le r$ for which  the words  $R_1 , \dots, R_m$ are \crd .

Conversely, suppose that the CAC-conjecture holds for every presentation \eqref{e1} for which the words $R_1, \dots, R_m$ are \crd\ and  $m \le r$. Consider an arbitrary  presentation
\begin{equation}\label{e5}
 \langle \, a_1,  \dots, a_m  \,  \|  \, W_1 , \dots, W_m  \, \rangle
\end{equation}
of the trivial group, where $W_1 , \dots, W_m$ are reduced words over $\A^{\pm 1}$. Let $\bar W_1 , \dots, \bar W_m$ be \crd\ words obtained from $W_1 , \dots, W_m$, resp., by cyclic cancellations. Since the CAC-conjecture holds for the presentation
\begin{equation*}
  \langle \, a_1,  \dots, a_m  \,  \|  \, \bar W_1 , \dots, \bar W_m  \, \rangle ,
\end{equation*}
it follows that there is a finite sequence of operations (CT1)--(CT3) that changes the tuple
$(\bar W_1$, $\dots,  \bar W_m )$ into $(a_1,  \dots, a_m)$. Note that every operation of type (CT1)--(CT3) over a \crd\ tuple  can be presented as a composition of operations
 (T1)--(T3). Therefore, the tuple $(\bar W_1 , \dots, \bar W_m )$ can also be converted into $(a_1,  \dots, a_m)$ by a sequence of operations (T1)--(T3). Since $(W_1 , \dots, W_m )$ can be turned into  $(\bar W_1 , \dots, \bar W_m )$ by operations (T3), the AC-conjecture is also true for presentation~\eqref{e5}.
\end{proof}

\begin{proof}[Proof of Corollary~\ref{c2}]   This is straightforward from  Corollary~\ref{c1}.
\end{proof}

\begin{proof}[Proof of Corollary~\ref{c3}] Let $r \ge 2$ be an integer and assume that
  every  $m$-tuple $\R =(R_1, \dots, R_m)$,
where $2 \le m \le r$,
that defines the trivial group by \eqref{e1},  can be transformed to  $(a_1, \dots, a_m)$ by a finite sequence of  operations  (T1)--(T3).  Let $\bar R_1 , \dots, \bar R_m$ be \crd\ words obtained from $R_1 , \dots, R_m$, resp., by cyclic cancellations. It follows from  Corollary~\ref{c2} that the $m$-tuple $\bar \R =(\bar R_1, \dots, \bar R_m)$ can be turned into
$(a_1, \dots, a_m)$ by  operations  (CT1)--(CT3). Note that every operation of type (CT1)--(CT3) over a \crd\ tuple  can be presented as a composition of operations
 (T1), (T2), (T3C). Hence, the tuple $(\bar R_1 , \dots, \bar R_m )$ can also be converted into $(a_1,  \dots, a_m)$ by a sequence of operations (T1), (T2), (T3C).  It remains to observe that the original $m$-tuple $(R_1 , \dots, R_m )$ can be turned into  $(\bar R_1 , \dots, \bar R_m )$ by operations (T3C).
\end{proof}

\section{One More Conjecture of Andrews and Curtis}

Here we discuss one more satellite hypothesis of Andrews and Curtis   \cite[Conjecture~4]{AC66}  concerning nonminimal pairs of words. According
to \cite{AC66},  a pair $(W_1 ,  W_2 )$ of reduced words $W_1,  W_2$ over the alphabet $\{ a^{\pm 1},  b^{\pm 1} \}$ is called {\em minimal} if no sequence of operations (T1)--(T3)  can decrease the total length $|W_1|+ |W_2|$ of  $(W_1, W_2 )$.

In  \cite[Conjecture~4]{AC66}, Andrews and Curtis speculate that  if $(W_1 ,  W_2 )$ is not a minimal pair,
then $W_1$ and $W_2$, considered as cyclic words, contain a common subword $V$ such that
\begin{equation}\label{e6}
 | W_1 | + | W_2 | -2 |V| < \max (|W_1|, |W_2|) .
\end{equation}
In other words, there are cyclic permutations $\bar W_1$, $\bar W_2$ of \crd\ words
$W_1^{\e_1}$, $W_2^{\e_2}$, where $\e_1, \e_2 = \pm 1$, such that $\bar W_1 \equiv V_1 V $, $\bar W_2 \equiv V^{-1} V_2 $ and the
product $\bar W_1  \bar W_2 \equiv V_1 V_2$ is \crd\ and shorter than a longest of $W_1, W_2$. This conjecture would provide a strong Nielsen-type reduction for nonminimal pairs.
However, the conjecture is false and, as a counterexample, one could use the pair
$(a^2 b^{-3}, aba  b^{-1} a^{-1} b^{-1} )$.
For this pair, if $V$ is a common subword of cyclic permutations of  words $W_1^{\e_1}$ and  $W_2^{\e_2}$, where $\e_1, \e_2 = \pm 1$,  then it is easy to see that $|V| \le 2$. Hence, the inequality  \eqref{e6} could not be satisfied.
On the other hand, as was found out by Myasnikov \cite{M}, see also   \cite{HR}, \cite{MMS02}, the AC-conjecture holds for the pair  $(a^2 b^{-3}, aba b^{-1} a^{-1} b^{-1} )$. Therefore, the pair
$(a^2 b^{-3}, aba b^{-1} a^{-1} b^{-1} )$ is not minimal and gives a counterexample to
\cite[Conjecture~4]{AC66}.

\section{Cancellative Cyclic  Version of the Andrews--Curtis Conjecture}

The significance and power of stabilizations does not look clear even in the special case of presentations coming from spines  of the  3-sphere and is totally obscure for arbitrary presentations.  For this reason,  it seems worthwhile to consider a more restrictive version of the CAC-conjecture with and without stabilizations, called the cancellative cyclic  version of the Andrews--Curtis conjecture and abbreviated as CCAC-conjecture. In this new version, in the analogue of operation (CT2) we require  complete cancellation of one of the words $W_i, W_j$ in the cyclic product  $W_i W_j$.  This CCAC-conjecture enables us to give the first evidence of  importance of stabilizations in the context of the AC-conjecture.  We will show in Theorem~\ref{t2} that the CCAC-conjecture with stabilizations is still equivalent to the AC-conjecture with stabilizations, whereas the CCAC-conjecture {\em without} stabilizations is false.

\medskip

As before, let $\W = (W_1, \dots, W_n)$ be a \crd\ $n$-tuple of words over $\A^{\pm 1}$. Consider the following transformation over $\W$.
\begin{enumerate}
\item[(CCT2)]  For some pair of distinct indices $i$ and $j$, $W_i$ is replaced with a word $W$,
where $W$ is a cyclically reduced or empty word obtained from the product
$ W_i W_j$ by making cancellations and cyclic cancellations and $W$ is such that
   $|W|  \le \max (|W_i|, |W_j|) - \min (|W_i|, |W_j|)$.
\end{enumerate}

It is easy to see that the latter inequality  is equivalent to the condition that one of the words $W_i, W_j$ cancels out completely in the cyclic product $W_i W_j$.  In particular, this means that the analogue of operation (T4) over \crd\ tuples would be meaningless when combined with operations (CT1), (CCT2), (CT3). Indeed, if one of the words
$W_i$, $W_j$ is a letter $b \not\in \A^{\pm 1}$ and the other one is a word over $\A^{\pm 1}$ then  (CCT2) would not be applicable to the pair $W_i$, $W_j$.
For this reason, a suitable analogue of operation (T4) over pairs $\A, \R$, where $\R$ is a \crd\  tuple of words  over $\A^{\pm 1}$,  is defined as follows.

\begin{enumerate}
\item[(CCT4)]
Let $b \not\in \A^{\pm 1}$ be a letter and let $U$ be a word over $\A^{\pm 1}$.
Add $b$ to the alphabet $\A$ and append the word $bU$ to the tuple $\R$.
Conversely, if $bU$ is a word of $\R$, where $b \in \A$, and $b,  b^{-1}$
have no  occurrences  in $U$ and in all  words of $\R$ other than $bU$, then
delete $b$ from $\A$ and delete $bU$ from $\R$.
\end{enumerate}

The cancellative cyclic version of the  Andrews--Curtis conjecture,  abbreviated as {\em CCAC-conjecture},
states that, {\em for every balanced group presentation}
\begin{equation}\label{e1a}
\PP = \langle \, a_1,  \dots, a_m  \,  \|  \, R_1 , \dots, R_m  \, \rangle
\end{equation}
{\em such that \eqref{e1a}    defines the trivial group and $\R = (R_1,
\dots, R_m)$  is cyclically reduced, the $m$-tuple  $\R $ can be brought to the letter tuple $(a_1, \dots, a_m)$ by a finite sequence of \trn  s   (CT1), (CCT2), (CT3)}.

Similarly, the cancellative cyclic version of the  Andrews--Curtis conjecture with stabilizations, briefly {\em CCAC-conjecture with stabilizations},
claims that, {\em for every balanced group presentation  \eqref{e1a}
such that \eqref{e1a}    defines the trivial group  and $\R = (R_1,
\dots, R_m)$  is cyclically reduced,   the $m$-tuple    $\R$ can be brought to the letter tuple $(a_1, \dots, a_m)$ by a finite sequence of \trn  s   (CT1), (CCT2), (CT3), (CCT4)}.

\begin{thm}\label{t2} $\rm{(a)}$   Suppose that  a balanced  presentation \eqref{e1a}  defines the trivial group and the tuple $\R = (R_1, \dots, R_m)$ is \crd .   Then  the  AC-conjecture with  stabilizations holds  true  for \eqref{e1a}   if and only if   the  CCAC-conjecture with  stabilizations holds  for \eqref{e1a}.

 $\rm{(b)}$   The  CCAC-conjecture {\em without}  stabilizations is false.
\end{thm}

\begin{proof} (a) First we will show that if the AC-conjecture  with  stabilizations holds  for \eqref{e1a}, then the CCAC-conjecture  with  stabilizations also holds  for \eqref{e1a}.

Assume that the AC-conjecture  with  stabilizations holds  for \eqref{e1a}
and $\R$ can be converted into $(a_1, \dots, a_m)$ by a sequence of
\Etrn s (T1)--(T3) and $2s$ stabilizations (T4). It is clear that, in this process of turning  $\R$ into  $(a_1, \dots, a_m)$, one can do all $s$
positive stabilizations (that increase $| \A |$) in the very beginning and all $s$ negative
stabilizations (or destabilizations that decrease $| \A |$) in the very end. Therefore, one can avoid
stabilizations altogether and assume that the $(m+s)$-tuple
$$
(\R, \B)  := (R_1, \ldots, R_m, b_1, \ldots, b_s) ,
$$
where $b_1, \dots, b_s$ are all new letters that were introduced by $s$ positive
stabilizations, can be converted into $(m+s)$-tuple $(\A, \B) := (a_1,
\dots, a_m, b_1, \dots, b_s)$ by a sequence of \Etrn s   of type (T1)--(T3).

Similarly to the proof of Theorem~\ref{t1}, let $\sigma_1, \dots, \sigma_\ell$ be operations of type (T1)--(T3)  that are applied to the $(m+s)$-tuple $(\R, \B)$  to obtain  $(\A, \B)$. Denote
\begin{equation*}
\W(0) := (\R, \B) \quad \ \mbox{ and} \quad \  \W(k)  :=   \sigma_k(\W(k-1))
\end{equation*}
for $k=1, \ldots, \ell$, hence,    $\W(k) =\sigma_k \dots \sigma_1(\W(0))$ and   $\W(\ell) = (\A, \B)$. Also, we denote
\begin{equation*}
\W(k) := (W_1(k), \dots, W_{m+s}(k)) \quad \mbox{ and} \quad \ \bar \W(k)
 := (\bar W_1(k), \dots, \bar W_{m+s}(k))  ,
\end{equation*}
where $\bar W_1(k), \dots, \bar W_{m+s}(k)$ are \crd\ words such that,
for every $i$, $\bar W_i(k)$ is obtained from $W_i(k)$ by cyclic cancellations, so
$W_i(k) \equiv  S_i(k) \bar W_i(k)  S_i(k)^{-1}$ in the free
group $\FF(\A \cup \B)$ with some word  $S_i(k)$ for $k = 0, \dots, \ell$.

By induction on $k \ge 0$, we will be proving that $\bar \W(k)$ can be
obtained from $\W(0)$ by a sequence of  operations (CT1), (CCT2), (CT3), (CCT4). Since
the basis step of this induction is obvious, we only need to make the
induction step from $k$ to $k+1$.

If $\sigma_{k+1}$ is of type (T1), then we can perform an analogous
operation (CT1) over $\bar \W(k)$ and obtain $\bar \W(k+1)$.
A reference to the induction hypothesis completes this case.

If $\sigma_{k+1}$ has type (T3), then no change is needed,  we can
set   $\bar \W(k+1) := \bar \W(k)$, and refer to the induction hypothesis.

Therefore, we may assume that $\sigma_{k+1}$ has type (T2) and $W_t(k+1) =
W_t(k)W_r(k)$  in  $\FF(\A\cup \B)$ with  $t \neq r$. To simplify notation, rename $U_i :=  \bar W_i(k)$, $i =1, \dots, m+s$, and $\U := \bar \W(k)$.
Reindexing if necessary, we may also suppose that $t=1$, $r =2$,  hence,
$W_1(k+1) = U_1 U_2$  in $\FF(\A\cup \B)$.

It follows from  the analogue of Lemma~\ref{Lem2} in which the word $\bar R_1(k+1)$ is replaced with $\bar W_1(k+1)$ that we need to consider Cases (F1)--(F4).

In Case (F3), we apply (CT3) to $U_1 $ to get $C^{-1} D C U_2^{-1}$ and use (CCT2) to
convert   $C^{-1} D C U_2^{-1}$ to $D$.  Since $D$ is a cyclic permutation of $\bar U_1(k+1)$, a reference to the induction hypothesis completes the induction step in Case (F3).

Case (F4) is analogous to Case (F3) with $U_1$ and $U_2$ switched.

It remains to study Cases (F1)--(F2).

Suppose that Case (F1) holds, hence, up to cyclic permutations of the words
$\bar W_1(k+1)$,  $U_1$, $U_2$, we have
$\bar W_1(k+1) \equiv U_1 C  U_2  C^{-1}$,  where $|C| > 0$, see  Fig.~1(a).
Applying  operations (CT3) to  $U_1$, $U_2$, $\bar W_1(k+1)$  if necessary, we may assume that   $\bar W_1(k+1) \equiv U_1 C U_2 C^{-1}$.

Now we apply a sequence of  operations  (CT1), (CCT2), (CT3), (CCT4) to
the tuple   $\U$   so that the first two components of $\U$    would be changing as indicated below. Note that the addition and  deletion of the third component is done by (CCT4)  and  $x$ is a letter,  $x \not\in (\A \cup \B)^{\pm 1}$.
\begin{multline*}
( U_1, U_2)
\to  (x  C U_2 C^{-1} U_1 ,  U_1, U_2)  \to  ( U_1^{-1}  C U_2^{-1} C^{-1}  x^{-1}  ,  U_1, U_2) \to  \\
(U_1^{-1}  C U_2^{-1} C^{-1}  x^{-1}  ,   C U_2^{-1} C^{-1}  x^{-1}  , U_2)  \to  (x  C U_2 C^{-1} U_1,   x  C U_2 C^{-1} ,  U_2^{-1})  \to \\
(x  C U_2 C^{-1} U_1 ,    C^{-1}  x  C U_2, U_2^{-1}) \to
(x  C U_2 C^{-1} U_1 ,     x , U_2^{-1})  \to  \\  ( C U_2 C^{-1} U_1 x ,     x^{-1} , U_2^{-1}) \to   ( C U_2 C^{-1} U_1 ,     x^{-1} , U_2^{-1}) \to  (  U_1 C U_2 C^{-1}, U_2)   .
\end{multline*}
Thus it is shown that  $\U = \bar W(k)$ can be changed into $\bar W(k+1)$  by operations
(CT1), (CCT2), (CT3), (CCT4). A reference to the induction hypothesis completes the induction step in Case (F1).

Assume that Case (F2) holds,  hence, up to cyclic permutations of the words
$\bar W_1(k+1)$,  $U_1$, $U_2$, we have $\bar W_1(k+1) \equiv D E$, where $U_1 \equiv D P^{-1}$,
$U_2 \equiv P E$, and  $|D|, |E| >0$,  see Fig.~1(b).  Applying operations  (CT3) to  $U_1$, $U_2$, $\bar W_1(k+1)$,  if necessary, we may assume that  $\bar W_1(k+1) \equiv D E$, where $U_1 \equiv D P^{-1}$, $U_2 \equiv P E$, and  $|D|, |E| >0$.

Let us  apply a sequence of  operations  (CT1), (CCT2), (CT3), (CCT4) to
the tuple   $\U$   so that the first two components of $\U$    would be changing as indicated below. As above,    $x \not\in (\A \cup \B)^{\pm 1}$  is a new letter.
\begin{multline*}
( U_1, U_2)  =   (DP^{-1},   PE ) \to  (x  PE DP^{-1}, DP^{-1}, PE )  \to \\
(  P D^{-1} E^{-1} P^{-1} x^{-1} , DP^{-1}, PE )  \to
(P D^{-1} E^{-1} P^{-1} x^{-1} ,  E^{-1} P^{-1} x^{-1}, PE )  \to \\
(x  PE DP^{-1}, x  PE ,    E^{-1} P^{-1} )  \to (x  PE DP^{-1}, x, E^{-1} P^{-1} ) \to \\
 (PE DP^{-1}x  , x^{-1},  P E ) \to  (E D  , x,  P E ) \to   (D E  ,   P E ) = (\bar W_1(k+1)  ,   U_2 ) .
\end{multline*}
Thus  $\U = \bar W(k)$ can be changed into $\bar W(k+1)$  by operations
(CT1), (CCT2), (CT3), (CCT4). A reference to the induction hypothesis completes the induction step in Case (F2).

The induction step is now complete in all Cases (F1)--(F4)   and it is shown that  $\bar \W(k)$,  for every $k\ge 0$,  can be obtained from $\W(0)= (\R, \B)$ by a sequence of  operations (CT1), (CCT2), (CT3), (CCT4). Since $\bar \W(\ell) = \W(\ell) = (\A, \B)$, it follows that, using
operations  (CCT4), one can transform the tuple $\R$ into
$ (\R, \B)$. Then, applying  operations    (CT1), (CCT2), (CT3), (CCT4), one can get $(\A, \B)$ from   $ (\R, \B)$,  and then, using  (CCT4),  obtain  $(a_1, \dots, a_m)$  from  $(\A, \B)$.  Thus the CCAC-conjecture with  stabilizations  holds for $\R$.

Conversely, assume that the CCAC-conjecture  with  stabilizations holds for $\R$. It is easy to see that every  operation (CT1), (CCT2), (CT3), (CCT4) is a composition of (T1)--(T4). Hence, the AC-conjecture with  stabilizations also holds for $\R$.
\smallskip

(b)   As a counterexample to the CCAC-conjecture, we  use the presentation
$$
\langle  \, a, b  \,  \|  \, a^2 b^{-3}, aba b^{-1} a^{-1} b^{-1}   \, \rangle ,
$$
where $ (a^2 b^{-3}, aba  b^{-1} a^{-1} b^{-1} ) =(W_1, W_2)$ is the pair of Sect.~3. As was observed in  Sect.~3,  if $V$ is a common subword of cyclic permutations of  words $W_1^{\e_1}$ and  $W_2^{\e_2}$, where $\e_1, \e_2 = \pm 1$,  then  $|V| \le 2 < \min(|W_1|,  |W_2|) = 5$. Therefore, no operation of type (CCT2) is applicable to any pair obtained from $(W_1, W_2)$ by a sequence of operations   (CT1), (CT3).
Since operations   (CT1), (CT3) do not change the length $|W_1 | + |W_2 |$, this proves that $(W_1, W_2)$ cannot be turned into  $(a, b)$ by  operations   (CT1), (CCT2), (CT3).  Theorem~\ref{t2} is proved.
\end{proof}

\smallskip

In conclusion, we recall that the Andrews--Curtis conjecture with stabilizations
is known to hold for presentations that come from spines of the 3-sphere and it would be of interest to find out whether there is an upper bound on the number of operations (T1)--(T4) in this situation. Note that such a computable bound  for spine  presentations associated with  3-manifolds, together with the 3-dimensional Poincar\'e conjecture,
would imply a purely algebraic algorithm to recognize the 3-sphere and to detect the triviality
of spine presentations associated with 3-manifolds.  It might be the case that available algorithms for recognition of the 3-sphere, together with analysis of their computational complexity, see \cite{I01}, \cite{I08}, \cite{Mt}, \cite{Rbn}, \cite{Tmp},  would be useful towards this goal.
\smallskip

{\em Acknowledgements.} The author wishes to thank the referee
for many meticulous remarks.

\end{document}